\documentclass[11pt,a4paper]{amsart}
\usepackage{eurosym}
\usepackage{graphicx,amscd,color,amsmath,amsfonts,amssymb,geometry,hyperref,soul}

\newtheorem{theorem}{Theorem}
\theoremstyle{plain}

\newtheorem{lemma}{Lemma}

\newtheorem{remark}{Remark}

\numberwithin{equation}{section}
\begin{document}
\title[Optimal Hardy--Littlewood inequalities uniformly bounded]{Optimal Hardy--Littlewood inequalities uniformly bounded by a
universal constant}
\author[Albuquerque]{N. Albuquerque}
\address[N. Albuquerque]{Departamento de Matem\'{a}tica \\
\indent Universidade Federal da Para\'{\i}ba \\
\indent 58.051-900 - Jo\~{a}o Pessoa, Brazil.}
\email{ngalbqrq@mat.ufpb.br}
\author[Ara\'{u}jo]{G. Ara\'{u}jo}
\address[G. Ara\'{u}jo]{Departamento de Matem\'{a}tica \\
\indent	Universidade Estadual da Para\'{\i}ba \\
\indent	58.429-500 - Campina Grande, Brazil.}
\email{gustavoaraujo@cct.uepb.edu.br}
\author[Maia]{M. Maia}
\address[M. Maia]{Departamento de Matem\'{a}tica \\
\indent Universidade Federal da Para\'{\i}ba \\
\indent 58.051-900 - Jo\~{a}o Pessoa, Brazil.}
\email{mariana.britomaia@gmail.com}
\author[Nogueira]{T. Nogueira}
\address[T. Nogueira]{Departamento de Matem\'{a}tica \\
\indent Universidade Federal da Para\'{\i}ba \\
\indent 58.051-900 - Jo\~{a}o Pessoa, Brazil.}
\email{tonykleverson@gmail.com}
\author[Pellegrino]{D. Pellegrino}
\address[D. Pellegrino]{Departamento de Matem\'{a}tica \\
\indent Universidade Federal da Para\'{\i}ba \\
\indent 58.051-900 - Jo\~{a}o Pessoa, Brazil.}
\email{pellegrino@pq.cnpq.br}
\author[Santos]{J. Santos}
\address[J. Santos]{Departamento de Matem\'{a}tica \\
\indent Universidade Federal da Para\'{\i}ba \\
\indent 58.051-900 - Jo\~{a}o Pessoa, Brazil.}
\email{joedson@mat.ufpb.br}
\thanks{D. Pellegrino is supported by CNPq}

\begin{abstract}
The Hardy--Littlewood inequality for $m$-linear forms on $\ell _{p}$ spaces
and $m<p\leq 2m$ asserts that
\begin{equation*}
\left( \sum_{j_{1},...,j_{m}=1}^{\infty }\left\vert T\left( e_{j_{1}},\ldots
,e_{j_{m}}\right) \right\vert ^{\frac{p}{p-m}}\right) ^{\frac{p-m}{p}}\leq
2^{\frac{m-1}{2}}\left\Vert T\right\Vert
\end{equation*}%
for all continuous $m$-linear forms $T:\ell _{p}\times \cdots \times \ell
_{p}\rightarrow \mathbb{R}$ or $\mathbb{C}.$ The case $m=2$ recovers a
classical inequality proved by Hardy and Littlewood in 1934. As a
consequence of the results of the present paper we show that the same
inequality is valid with $2^{\frac{m-1}{2}}$ replaced by $2^{\frac{\left(
m-1\right) \left( p-m\right) }{p}}$. In particular, for $m<p\leq m+1$ the
optimal constants of the above inequality are uniformly bounded by $2.$
\end{abstract}

\maketitle

\section{Introduction}

The famous Littlewood's $4/3$ inequality \cite{LLL}, proved in 1930, asserts
that%
\begin{equation*}
\left( \sum_{j,k=1}^{\infty }\left\vert T(e_{j},e_{k})\right\vert ^{\frac{4}{%
3}}\right) ^{\frac{3}{4}}\leq \sqrt{2}\left\Vert T\right\Vert
\end{equation*}%
for all continuous bilinear forms $T\colon c_{0}\times c_{0}\rightarrow
\mathbb{C}$, and the exponent $4/3$ cannot be improved. Besides its own
beauty, Littlewood's insights motivated further important works of
Bohnenblust and Hille (1931) and Hardy and Littlewood (1934).
Bohnenblust--Hille inequality \cite{bh} assures the existence of a constant $%
B_{m}\geq 1$ such that
\begin{equation}
\left( \sum_{j_{1},\cdots ,j_{m}=1}^{\infty }\left\vert T(e_{j_{1}},\cdots
,e_{j_{m}})\right\vert ^{\frac{2m}{m+1}}\right) ^{\frac{m+1}{2m}}\leq
B_{m}\left\Vert T\right\Vert ,  \label{u88}
\end{equation}%
for all continuous $m$--linear forms $T\colon c_{0}\times \cdots \times
c_{0}\rightarrow \mathbb{C}$. The case $m=2$ recovers Littlewood's $4/3$
inequality. Three years later, using quite delicate estimates, Hardy and
Littlewood \cite{hl} extended Littlewood's $4/3$ inequality to bilinear
forms defined on $\ell _{p}\times \ell _{q}.$ In 1981, Praciano-Pereira \cite%
{pra} extended the Hardy--Littlewood inequalities to $m$-linear forms on $%
\ell _{p}$ spaces for $p\geq 2m$ and quite recently Dimant and Sevilla-Peris
\cite{dimant} extended the estimates for the case $m<p\leq 2m$. These
results were extensively investigated in various directions in the recent
years (\cite{jfa-abps, isr-abps, ap2, ap, aaps, cap, dimant, nnn}). As a
matter of fact, all results hold for both real and complex scalars with
eventually different constants; from now on we denote $\mathbb{K}=\mathbb{R}$
or $\mathbb{C}$. In general terms we have the following $m$-linear
inequalities:

\begin{itemize}
\item If $p\geq 2m,$ then there are constants $B_{m,p}^{\mathbb{K}}\geq 1$
such that%
\begin{equation*}
\left( \sum_{j_{1},\cdots ,j_{m}=1}^{n}\left\vert T(e_{j_{1}},\cdots
,e_{j_{m}})\right\vert ^{\frac{2mp}{mp+p-2m}}\right) ^{\frac{mp+p-2m}{2mp}%
}\leq B_{m,p}^{\mathbb{K}}\left\Vert T\right\Vert
\end{equation*}%
for all $m$-linear forms $T:\ell _{p}^{n}\times \cdots \times \ell
_{p}^{n}\rightarrow \mathbb{K}$ and all positive integers $n$.

\item If $m<p \leq 2m$, then there are constants $B_{m,p}^{\mathbb{K}}\geq 1$
such that
\begin{equation*}
\left( \sum_{j_{1},...,j_{m}=1}^{n}\left\vert T\left( e_{j_{1}},\cdots
,e_{j_{m}}\right) \right\vert ^{\frac{p}{p-m}}\right) ^{\frac{p-m}{p}}\leq
B_{m,p}^{\mathbb{K}}\left\Vert T\right\Vert
\end{equation*}%
for all $m$-linear forms $T:\ell _{p}^{n}\times \cdots \times \ell
_{p}^{n}\rightarrow \mathbb{K}$ and all positive integers $n$.
\end{itemize}

The exponents of all above inequalities are optimal: if replaced by smaller
exponents the constants will depend on $n$. However, looking at the above
inequalities by an anisotropic viewpoint a much richer complexity arise
(see, for instance, \cite{jfa-abps, isr-abps, ap2, ap, cap, rp}).

The investigation of the sharp constants in above inequalities is more than
a puzzling mathematical challenge; for applications in physics we refer to
\cite{montanaro}. The first estimates for $B_{m,p}^{\mathbb{K}}$ had
exponential growth:
\begin{equation*}
B_{m,p}^{\mathbb{K}}\leq \left( \sqrt{2}\right) ^{m-1},
\end{equation*}%
for any $m\geq 1$. It was just quite recently that the estimates for $%
B_{m,p}^{\mathbb{K}}$ were refined, see for instance \cite{ap2, ap, bayart}
and references therein. It was proved in \cite{bayart} that
\begin{eqnarray}
B_{m,\infty }^{\mathbb{R}} &<&\kappa _{1}\cdot m^{\frac{2-\log 2-\gamma }{2}%
}\approx \kappa _{1}\cdot m^{0.36482}, \\
B_{m,\infty }^{\mathbb{C}} &<&\kappa _{2}\cdot m^{\frac{1-\gamma }{2}%
}\approx \kappa _{2}\cdot m^{0.21139},
\end{eqnarray}%
for certain constants $\kappa _{1},\kappa _{2}>0,$ where $\gamma $ is the
Euler-Mascheroni constant. For $p<\infty $, among other results it was shown
in \cite{ap} that for $p>2m(m-1)^{2}$ we have
\begin{equation*}
B_{m,p}^{\mathbb{K}}\leq B_{m,\infty }^{\mathbb{K}}.
\end{equation*}%
The best known estimates of $B_{m,p}^{\mathbb{K}}$ for the case $m<p\leq 2m$
are $\left( \sqrt{2}\right) ^{m-1}$ (see \cite{isr-abps, dimant}). These
estimates (case $m<p\leq 2m$) are somewhat intriguing. In fact, if $p=m$ it
is easy to show that the only Hardy--Littlewood type inequality
\begin{equation*}
\left( \sum_{j_{1},\cdots ,j_{m}=1}^{n}\left\vert T(e_{j_{1}},\cdots
,e_{j_{m}})\right\vert ^{s}\right) ^{\frac{1}{s}}\leq B_{m,m}^{\mathbb{K}%
}\left\Vert T\right\Vert
\end{equation*}%
happens for $s=\infty $ (of course, here we consider the $\sup $ norm) and
in this case it is obvious that the optimal constants are $B_{m,m}^{\mathbb{K%
}}=1.$ So, we have optimal constants equal to $1$ for $p=m$ and the best
known constants $\left( \sqrt{2}\right) ^{m-1}$ for $p$ close to $m.$ In
this paper, among other results, we show that in fact the estimates $\left(
\sqrt{2}\right) ^{m-1}$ are far from being optimal: we prove that%
\begin{equation*}
B_{m,p}^{\mathbb{K}}\leq 2^{\frac{\left( m-1\right) \left( p-m\right) }{p}}.
\end{equation*}

We present below the estimate obtained by Dimant and Sevilla-Peris (\cite%
{dimant}) for further reference:

\begin{theorem}[Dimant and Sevilla-Peris]
\label{i99}Let $m\geq 2$ be a positive integer and \ $p_{j}>1$ for all $j$
and%
\begin{equation*}
\frac{1}{2}\leq \frac{1}{p_{1}}+\cdots +\frac{1}{p_{m}}<1.
\end{equation*}%
Then%
\begin{equation*}
\left( \sum_{j_{1},...,j_{m}=1}^{n}\left\vert T\left( e_{j_{1}},\ldots
,e_{j_{m}}\right) \right\vert ^{\frac{1}{1-\left( \frac{1}{p_{1}}+\cdots +%
\frac{1}{p_{m}}\right) }}\right) ^{1-\left( \frac{1}{p_{1}}+\cdots +\frac{1}{%
p_{m}}\right) }\leq \left( \sqrt{2}\right) ^{m-1}\left\Vert T\right\Vert ,
\end{equation*}%
for all $m$-linear forms $T:\ell _{p_{1}}^{n}\times \cdots \times \ell
_{p_{m}}^{n}\rightarrow \mathbb{K}$ and all positive integers $n$. In
particular, if $m<p\leq 2m$, then,
\begin{equation*}
\left( \sum_{j_{1},...,j_{m}=1}^{n}\left\vert T\left( e_{j_{1}},\ldots
,e_{j_{m}}\right) \right\vert ^{\frac{p}{p-m}}\right) ^{\frac{p-m}{p}}\leq
\left( \sqrt{2}\right) ^{m-1}\left\Vert T\right\Vert
\end{equation*}%
for all continuous $m$-linear forms $T:\ell _{p}\times \cdots \times \ell
_{p}\rightarrow \mathbb{K}$.
\end{theorem}

The exponent $\frac{1}{1-\left( \frac{1}{p_{1}}+\cdots +\frac{1}{p_{m}}%
\right) }$ is optimal, but if one works in the anisotropic setting the
result is not optimal (see, for instance, \cite{aaps, rp}). The main results
of the present paper are the forthcoming Theorems \ref{76543}, \ref{765432}
and \ref{7654321} which also improve the original constants of the bilinear
Hardy--Littlewood inequalities. For instance, for $m<p\leq m+1$ the optimal
constants of the Hardy--Littlewood inequality are uniformly bounded by $2.$

\section{A multipurpose lemma}

Let $m\geq 2$ be a positive integer, $F$ be a Banach space, $A\subset
I_{m}:=\{1,\dots ,m\},p_{1},\dots ,p_{m},s,\alpha \geq 1$ and
\begin{equation*}
B_{p_{1},\dots ,p_{m}}^{A,s,\alpha ,F,n} :=\inf \left\{ C(n)\geq 0 \,: \
\left( \sum_{j_{i}=1}^{n} \left( \sum_{\widehat{j_{i}}=1}^{n} \left|
T\left(e_{j_{1}},\dots ,e_{j_{m}}\right) \right|^{s} \right)^{\frac{1}{s}%
\alpha} \right) ^{\frac{1}{\alpha }} \leq C(n), \text{ for all } i\in A
\right\},
\end{equation*}
in which $\widehat{j_{i}}$ means that the sum runs over all indexes but $%
j_{i}$, and the infimum is taken over all norm-one $m$-linear operators $%
T:\ell _{p_{1}}^{n}\times \cdots \times \ell_{p_{m}}^{n} \to F$. The
following lemma -- fundamental in the proof of our main results -- is based on ideas dating back to Hardy and Littlewood (see
\cite{hl} and \cite{pra}), and we believe that it is of independent interest:

\begin{lemma}
\label{geral} Let $1\leq p_{k}<q_{k}\leq \infty ,\,k=1,\dots ,m$ and $%
\lambda _{0},s\geq 1$.

(a) If
\begin{equation}
\sum_{j=1}^{m}\left( \frac{1}{p_{j}}-\frac{1}{q_{j}}\right) <\frac{1}{%
\lambda _{0}}\quad \text{ and }\quad s\geq \left[ \frac{1}{\lambda _{0}}%
-\sum_{j=1}^{m}\left( \frac{1}{p_{j}}-\frac{1}{q_{j}}\right) \right]
^{-1}=:\eta _{1},  \label{hipinica}
\end{equation}%
then
\begin{equation*}
B_{p_{1},\dots ,p_{m}}^{I_{m},s,\eta _{1},F,n}\leq B_{q_{1},\dots
,q_{m}}^{I_{m},s,\lambda _{0},F,n}.
\end{equation*}%
\noindent (b) If%
\begin{equation}
\sum_{j=1}^{m}\left( \frac{1}{p_{j}}-\frac{1}{q_{j}}\right) <\frac{1}{%
\lambda _{0}}\quad \text{ and }\quad s\geq \left[ \frac{1}{\lambda _{0}}%
-\sum_{j=1}^{m-1}\left( \frac{1}{p_{j}}-\frac{1}{q_{j}}\right) \right]
^{-1}=:\eta _{2}  \label{hipinicb}
\end{equation}%
then
\begin{equation*}
B_{p_{1},\dots ,p_{m}}^{\{m\},s,\eta _{2},F,n}\leq B_{q_{1},\dots
,q_{m}}^{I_{m},s,\lambda _{0},F,n}.
\end{equation*}
\end{lemma}

\begin{proof}
To prove (a), let $s,\lambda _{0}$ be such that \eqref{hipinica} is
fulfilled. Let us define
\begin{equation*}
\lambda _{j}:=\left[ \frac{1}{\lambda _{0}}-\sum_{i=1}^{j}\left( \frac{1}{%
p_{i}}-\frac{1}{q_{i}}\right) \right] ^{-1},\quad j=1,\dots ,m.
\end{equation*}%
Notice that $\lambda _{m}=\eta _{1}$,
\begin{equation*}
\lambda _{j-1}<\lambda _{j}\quad \text{ and }\quad \left[ \frac{q_{j}p_{j}}{%
\lambda _{j-1}(q_{j}-p_{j})}\right] ^{\ast }=\frac{\lambda _{j}}{\lambda
_{j-1}},\quad \text{ for all }j=1,\dots ,m.
\end{equation*}%
Let us suppose that, for $k\in \{1,\dots ,m\}$,
\begin{equation}
\left( \sum_{j_{i}=1}^{n}\left( \sum_{\widehat{j_{i}}=1}^{n}\left\vert
T(e_{j_{1}},...,e_{j_{m}})\right\vert ^{s}\right) ^{\frac{1}{s}\lambda
_{k-1}}\right) ^{\frac{1}{\lambda _{k-1}}}\leq B_{q_{1},\dots
,q_{m}}^{I_{m},s,\lambda _{0},F,n}\Vert T\Vert   \label{Hipind1}
\end{equation}%
is true for all continuous $m$--linear operators $T:\ell _{p_{1}}^{n}\times
\cdots \times \ell _{p_{k-1}}^{n}\times \ell _{q_{k}}^{n}\times \cdots
\times \ell _{q_{m}}^{n}\rightarrow F$ and for all $i=1,...,m$. Let us prove
that
\begin{equation}
\left( \sum_{j_{i}=1}^{n}\left( \sum_{\widehat{j_{i}}=1}^{n}\left\vert
T(e_{j_{1}},...,e_{j_{m}})\right\vert ^{s}\right) ^{\frac{1}{s}\lambda
_{k}}\right) ^{\frac{1}{\lambda _{k}}}\leq B_{q_{1},\dots
,q_{m}}^{I_{m},s,\lambda _{0},F,n}\Vert T\Vert   \label{destese}
\end{equation}%
for all continuous $m$--linear operators $T:\ell _{p_{1}}^{n}\times \cdots
\times \ell _{p_{k}}^{n}\times \ell _{q_{k+1}}^{n}\times \cdots \times \ell
_{q_{m}}^{n}\rightarrow F$ and for all $i=1,...,m$. The first induction step
is our hypothesis. Consider
\begin{equation*}
T:\ell _{p_{1}}^{n}\times \dots \times \ell _{p_{k}}^{n}\times \ell
_{q_{k+1}}^{n}\times \dots \times \ell _{q_{m}}^{n}\rightarrow F,
\end{equation*}%
a $m$-linear operator and, for each $x\in B_{\ell _{\frac{q_{k}p_{k}}{%
q_{k}-p_{k}}}^{n}}$ define
\begin{equation*}
\begin{array}{ccccl}
T^{(x)} & : & \ell _{p_{1}}^{n}\times \cdots \times \ell
_{p_{k-1}}^{n}\times \ell _{q_{k}}^{n}\times \cdots \times \ell _{q_{m}}^{n}
& \rightarrow  & F \\
&  & (z^{(1)},\dots ,z^{(m)}) & \mapsto  & T(z^{(1)},\dots
,z^{(k-1)},xz^{(k)},z^{(k+1)},\dots ,z^{(m)}),%
\end{array}%
\end{equation*}%
with $xz^{(k)}=(x_{j}z_{j}^{(k)})_{j=1}^{n}\in \ell _{p_{k}}^{n}$. Observe
that
\begin{equation*}
\Vert T\Vert \geq \sup \left\{ \Vert T^{(x)}\Vert :x\in B_{\ell _{\frac{%
q_{k}p_{k}}{q_{k}-p_{k}}}^{n}}\right\} .
\end{equation*}%
By applying the induction hypothesis to $T^{(x)}$, we obtain
\begin{align}
& \left( \sum_{j_{i}=1}^{n}\left( \sum_{\widehat{j_{i}}=1}^{n}\left\vert
T\left( e_{j_{1}},...,e_{j_{m}}\right) \right\vert ^{s}\left\vert
x_{j_{k}}\right\vert ^{s}\right) ^{\frac{1}{s}\lambda _{k-1}}\right) ^{\frac{%
1}{\lambda _{k-1}}}  \notag \\
& =\left( \sum_{j_{i}=1}^{n}\left( \sum_{\widehat{j_{i}}=1}^{n}\left\vert
T\left(
e_{j_{1}},...,e_{j_{k-1}},xe_{j_{k}},e_{j_{k+1}},...,e_{j_{m}}\right)
\right\vert ^{s}\right) ^{\frac{1}{s}\lambda _{k-1}}\right) ^{\frac{1}{%
\lambda _{k-1}}}  \notag \\
& =\left( \sum_{j_{i}=1}^{n}\left( \sum_{\widehat{j_{i}}=1}^{n}\left\vert
T^{(x)}\left( e_{j_{1}},...,e_{j_{m}}\right) \right\vert ^{s}\right) ^{\frac{%
1}{s}\lambda _{k-1}}\right) ^{\frac{1}{\lambda _{k-1}}}  \notag \\
& \leq B_{q_{1},\dots ,q_{m}}^{I_{m},s,\lambda _{0},F,n}\Vert T^{(x)}\Vert
\notag \\
& \leq B_{q_{1},\dots ,q_{m}}^{I_{m},s,\lambda _{0},F,n}\Vert T\Vert
\label{guga0301}
\end{align}%
for all $i=1,...,m.$

Since
\begin{equation*}
\left[\frac{q_jp_j}{\lambda_{j-1}(q_j-p_j)}\right]^*= \frac{\lambda_j}{%
\lambda_{j-1}},
\end{equation*}
for all $j=1,\dots, m$, we have
\begin{align*}
& \left( \sum_{j_{k}=1}^{n}\left( \sum_{\widehat{j_{k}}=1}^{n}\left\vert
T\left( e_{j_{1}},...,e_{j_{m}}\right) \right\vert^{s}\right) ^{\frac{1}{s}%
\lambda_{k}}\right) ^{\frac{1}{\lambda_{k}}} \\
& = \left( \sum_{j_{k}=1}^{n}\left( \sum_{\widehat{j_{k}}=1}^{n}\left\vert
T\left(e_{j_{1}},...,e_{j_{m}}\right)\right\vert^{s}\right)^{\frac{1}{s}%
\lambda_{k-1}\left[\frac{q_kp_k}{\lambda_{k-1}(q_k-p_k)}\right]^*}\right)^{%
\frac{1}{\lambda_{k-1}}\cdot\frac{1}{\left[\frac{q_kp_k}{%
\lambda_{k-1}(q_k-p_k)}\right]^*}} \\
& =\left\Vert \left( \left( \sum_{\widehat{j_{k}}=1}^{n}\left\vert T\left(
e_{j_{1}},...,e_{j_{m}}\right) \right\vert^{s}\right) ^{\frac{1}{s}%
\lambda_{k-1}}\right) _{j_{k}=1}^{n}\right\Vert_{\left[\frac{q_kp_k}{%
\lambda_{k-1}(q_k-p_k)}\right]^*}^{\frac{1}{\lambda_{k-1}}} \\
& =\left( \sup_{y\in B_{\ell_{\frac{q_kp_k}{\lambda_{k-1}(q_k-p_k)}%
}^{n}}}\sum_{j_{k}=1}^{n}|y_{j_{k}}|\left( \sum_{\widehat{j_{k}}%
=1}^{n}\left\vert T\left( e_{j_{1}},...,e_{j_{m}}\right)
\right\vert^{s}\right)^{\frac{1}{s}\lambda_{k-1}}\right) ^{\frac{1}{%
\lambda_{k-1}}} \\
& =\left( \sup_{x\in B_{\ell_{\frac{q_kp_k}{q_k-p_k}}^{n}}}%
\sum_{j_{k}=1}^{n}|x_{j_{k}}|^{\lambda_{k-1}}\left( \sum_{\widehat{j_{k}}%
=1}^{n}\left\vert T\left(e_{j_{1}},...,e_{j_{m}}\right)
\right\vert^{s}\right)^{\frac{1}{s}\lambda_{k-1}}\right) ^{\frac{1}{%
\lambda_{k-1}}} \\
& = \sup_{x\in B_{\ell_{\frac{q_kp_k}{q_k-p_k}}^{n}}}\left(%
\sum_{j_{k}=1}^{n}\left(\sum_{\widehat{j_{k}}=1}^{n}\left\vert
T\left(e_{j_{1}},...,e_{j_{m}}\right)\right\vert^{s}\left\vert
x_{j_{k}}\right\vert^{s}\right)^{\frac{1}{s}\lambda_{k-1}}\right)^{\frac{1}{%
\lambda_{k-1}}} \\
& \leq B_{q_{1},\dots,q_{m}}^{I_{m},s,\lambda _{0},F,n} \Vert T\Vert.
\end{align*}
This proves \eqref{destese} for $i=k$. To prove \eqref{destese} for $i \neq
k $ let us consider initially $k\neq m$. Define
\begin{equation*}
S_{i}=\left(\sum_{\widehat{j_{i}}=1}^{n}|T(e_{j_{1}},...,e_{j_{m}})|^{s}%
\right)^{\frac{1}{s}}, \qquad i=1,....,m.
\end{equation*}

Note that
\begin{align*}
\sum_{j_{i}=1}^{n}&\left(\sum_{\widehat{j_{i}}%
=1}^{n}|T(e_{j_{1}},...,e_{j_{m}})|^{s}\right)^{\frac{1}{s}\lambda_{k}} \\
&
=\sum_{j_{i}=1}^{n}S_{i}^{\lambda_{k}}=\sum_{j_{i}=1}^{n}S_{i}^{%
\lambda_{k}-s}S_{i}^{s} \\
& = \sum_{j_{i}=1}^{n}\sum_{\widehat{j_{i}}=1}^{n}\frac{%
|T(e_{j_{1}},...,e_{j_{m}})|^{s}}{S_{i}^{s-\lambda_{k}}} \\
& = \sum_{j_{k}=1}^{n}\sum_{\widehat{j_{k}}=1}^{n}\frac{%
|T(e_{j_{1}},...,e_{j_{m}})|^{s}}{S_{i}^{s-\lambda_{k}}} \\
& = \sum_{j_{k}=1}^{n}\sum_{\widehat{j_{k}}=1}^{n}\frac{%
|T(e_{j_{1}},...,e_{j_{m}})|^{\frac{s(s-\lambda_{k})}{s-\lambda_{k-1}}}}{%
S_{i}^{s-\lambda_{k}}}|T(e_{j_{1}},...,e_{j_{m}})|^{\frac{%
s(\lambda_{k}-\lambda_{k-1})}{s-\lambda_{k-1}}}.
\end{align*}

From H\"{o}lder's inequality (first with exponents $r=\frac{s-\lambda _{k-1}%
}{s-\lambda _{k}}$ and $r^{\ast }=\frac{s-\lambda _{k-1}}{\lambda
_{k}-\lambda _{k-1}}$ and then with exponents $r=\frac{\lambda _{k}\left(
s-\lambda _{k-1}\right) }{\lambda_{k-1}\left( s-\lambda _{k}\right) }$ and $%
r^{\ast }=\frac{\lambda _{k}\left(s-\lambda _{k-1}\right) }{s\left(
\lambda_{k}-\lambda _{k-1}\right) }$), we have
\begin{align}
& \sum_{j_{i}=1}^{n}\left( \sum_{\widehat{j_{i}}%
=1}^{n}|T(e_{j_{1}},...,e_{j_{m}})|^{s}\right) ^{\frac{1}{s}\lambda_{k}}
\notag \\
& = \sum_{j_{k}=1}^{n}\sum_{\widehat{j_{k}}=1}^{n}\frac{%
|T(e_{j_{1}},...,e_{j_{m}})|^{\frac{s(s-\lambda_{k})}{s-\lambda_{k-1}}}}{%
S_{i}^{s-\lambda_{k}}}|T(e_{j_{1}},...,e_{j_{m}})|^{\frac{%
s(\lambda_{k}-\lambda_{k-1})}{s-\lambda_{k-1}}}  \notag \\
& \leq \sum_{j_{k}=1}^{n}\left( \sum_{\widehat{j_{k}}=1}^{n} \frac{%
|T(e_{j_{1}},...,e_{j_{m}})|^{s}}{S_{i}^{s-\lambda_{k-1}}}\right)^{\frac{%
s-\lambda_{k}}{s-\lambda_{k-1}}}\left( \sum_{\widehat{j_{k}}%
=1}^{n}|T(e_{j_{1}},...,e_{j_{m}})|^{s}\right) ^{\frac{\lambda_{k}-%
\lambda_{k-1}}{s-\lambda_{k-1}}}  \notag \\
& \leq\left( \sum_{j_{k}=1}^{n}\left( \sum_{\widehat{j_{k}}=1}^{n}\frac{%
|T(e_{j_{1}},...,e_{j_{m}})|^{s}}{S_{i}^{s-\lambda_{k-1}}}\right)^{\frac{%
\lambda_{k}}{\lambda_{k-1}}}\right)^{\frac{\lambda_{k-1}}{\lambda_{k}}\cdot%
\frac{s-\lambda_{k}}{s-\lambda_{k-1}}}  \notag \\
& \qquad \times \left( \sum_{j_{k}=1}^{n}\left( \sum_{\widehat{j_{k}}%
=1}^{n}|T(e_{j_{1}},...,e_{j_{m}})|^{s}\right) ^{\frac{1}{s}%
\lambda_{k}}\right) ^{\frac{1}{\lambda_{k}}\cdot\frac{(\lambda_{k}-%
\lambda_{k-1})s}{s-\lambda_{k-1}}}  \label{ant}
\end{align}
Let us estimate separately the two factors of this product. It follows from
the case $i=k$ that
\begin{equation}  \label{huhi}
\left( \sum_{j_{k}=1}^{n}\left( \sum_{\widehat{j_{k}}%
=1}^{n}|T(e_{j_{1}},...,e_{j_{m}})|^{s}\right)^{\frac{1}{s}%
\lambda_{k}}\right)^{\frac{1}{\lambda_{k}}\cdot\frac{(\lambda_{k}-%
\lambda_{k-1})s}{s-\lambda_{k-1}}}\leq\left(B_{q_{1},\dots,q_{m}}^{I_{m},s,%
\lambda _{0},F,n}\Vert T\Vert\right) ^{\frac{(\lambda_{k}-\lambda_{k-1})s}{%
s-\lambda_{k-1}}}.
\end{equation}
For the first factor, from H\"{o}lder's inequality with exponents $r=\frac{s%
}{s-\lambda _{k-1}}$ and $r^*=\frac{s}{\lambda _{k-1}}$ and the induction
hypothesis, we get
\begin{align*}
&\left( \sum_{j_{k}=1}^{n}\left( \sum_{\widehat{j_{k}}=1}^{n}\frac{%
|T(e_{j_{1}},...,e_{j_{m}})|^{s}}{S_{i}^{s-\lambda_{k-1}}}\right)^{\frac{%
\lambda_{k}}{\lambda_{k-1}}}\right)^{\frac{\lambda_{k-1}}{\lambda_{k}}} \\
& = \left\Vert \left( \sum_{\widehat{j_{k}}}\frac{%
|T(e_{j_{1}},...,e_{j_{m}})|^{s}}{S_{i}^{s-\lambda_{k-1}}}%
\right)_{j_{k}=1}^{n}\right\Vert_{\left[\frac{q_kp_k}{\lambda_{k-1}(q_k-p_k)}%
\right]^*} \\
& =\sup_{y\in B_{\ell_{\frac{q_kp_k}{\lambda_{k-1}(q_k-p_k)}%
}^{n}}}\sum_{j_{k}=1}^{n}|y_{j_{k}}|\sum_{\widehat{j_{k}}=1}^{n}\frac{%
|T(e_{j_{1}},...,e_{j_{m}})|^{s}}{S_{i}^{s-\lambda_{k-1}}} \\
& =\sup_{x\in B_{\ell_{\frac{q_kp_k}{q_k-p_k}}^{n}}}%
\sum_{j_{k}=1}^{n}|x_{j_{k}}|^{\lambda_{k-1}}\sum_{\widehat{j_{k}}=1}^{n}%
\frac{|T(e_{j_{1}},...,e_{j_{m}})|^{s}}{S_{i}^{s-\lambda_{k-1}}} \\
& =\sup_{x\in B_{\ell_{\frac{q_kp_k}{q_k-p_k}}^{n}}}\sum_{j_{k}=1}^{n}\sum_{%
\widehat{j_{k}}=1}^{n}\frac{|T(e_{j_{1}},...,e_{j_{m}})|^{s}}{%
S_{i}^{s-\lambda_{k-1}}}|x_{j_{k}}|^{\lambda_{k-1}} \\
& =\sup_{x\in B_{\ell_{\frac{q_kp_k}{q_k-p_k}}^{n}}}\sum_{j_{i}=1}^{n}\sum_{%
\widehat{j_{i}}=1}^{n}\frac{|T(e_{j_{1}},...,e_{j_{m}})|^{s}}{%
S_{i}^{s-\lambda_{k-1}}}|x_{j_{k}}|^{\lambda_{k-1}} \\
& =\sup_{x\in B_{\ell _{\frac{q_kp_k}{q_k-p_k}}^{n}}}\sum_{j_{i}=1}^{n}\sum_{%
\widehat{j_{i}}=1}^{n}\frac{|T(e_{j_{1}},...,e_{j_{m}})|^{s-\lambda_{k-1}}}{%
S_{i}^{s-\lambda_{k-1}}}|T(e_{j_{1}},...,e_{j_{m}})|^{%
\lambda_{k-1}}|x_{j_{k}}|^{\lambda_{k-1}} \\
& \leq \sup_{x\in B_{\ell _{\frac{q_kp_k}{q_k-p_k}}^{n}}}\sum_{j_{i}=1}^{n}%
\left(\sum_{\widehat{j_{i}}=1}^{n}\frac{|T(e_{j_{1}},...,e_{j_{m}})|^{s}}{%
S_{i}^{s}}\right) ^{\frac{s-\lambda_{k-1}}{s}}\left( \sum_{\widehat{j_{i}}%
=1}^{n}|T(e_{j_{1}},...,e_{j_{m}})|^{s}|x_{j_{k}}|^{s}\right)^{\frac{1}{s}%
\lambda_{k-1}} \\
& = \sup_{x\in B_{\ell _{\frac{q_kp_k}{q_k-p_k}}^{n}}}\sum_{j_{i}=1}^{n}%
\left(\sum_{\widehat{j_{i}}%
=1}^{n}|T(e_{j_{1}},...,e_{j_{m}})|^{s}|x_{j_{k}}|^{s}\right) ^{\frac{1}{s}%
\lambda_{k-1}} \\
& = \sup_{x\in B_{\ell_{\frac{q_kp_k}{q_k-p_k}}^{n}}}
\sum_{j_{i}=1}^{n}\left( \sum_{\widehat{j_{i}}=1}^{n}\left\vert
T^{(x)}\left(e_{j_{1}},...,e_{j_{m}}\right) \right\vert^{s}\right) ^{\frac{1%
}{s}\lambda_{k-1}} \\
& \leq \left( B_{q_{1},\dots,q_{m}}^{I_{m},s,\lambda _{0},F,n} \Vert T\Vert
\vspace{0.2cm}\right) ^{\lambda _{k-1}}.
\end{align*}
Therefore,
\begin{equation}  \label{fin2}
\left( \sum_{j_{k}=1}^{n}\left( \sum_{\widehat{j_{k}}=1}^{n}\frac{%
|T(e_{j_{1}},...,e_{j_{m}})|^{s}}{S_{i}^{s-\lambda_{k-1}}}\right)^{\frac{%
\lambda_{k}}{\lambda_{k-1}}}\right)^{\frac{\lambda _{k-1}}{\lambda_k}\cdot%
\frac{s-\lambda _{k}}{s-\lambda _{k-1} }}\leq \left(
B_{q_{1},\dots,q_{m}}^{I_{m},s,\lambda _{0},F,n}\Vert T\Vert \right)
^{\lambda _{k-1}\cdot \frac{s-\lambda _{k}}{s-\lambda _{k-1}}}.
\end{equation}
Replacing \eqref{huhi} and \eqref{fin2} in \eqref{ant} we finally conclude
that
\begin{align*}
& \sum_{j_{i}=1}^{n}\left( \sum_{\widehat{j_{i}}%
=1}^{n}|T(e_{j_{1}},...,e_{j_{m}})|^{s}\right) ^{\frac{1}{s}\lambda_{k}} \\
& \leq ( B_{q_{1},\dots,q_{m}}^{I_{m},s,\lambda _{0},F,n} \Vert T\Vert )
^{\lambda _{k-1}\cdot\frac{s-\lambda _{k}}{s-\lambda _{k-1}}}\cdot (
B_{q_{1},\dots,q_{m}}^{I_{m},s,\lambda _{0},F,n} \Vert T\Vert ) ^{\frac{%
\left( \lambda _{k}-\lambda _{k-1}\right) s}{ s-\lambda _{k-1}}} \\
& = (B_{q_{1},\dots,q_{m}}^{I_{m},s,\lambda _{0},F,n})^{\lambda _{k-1}\cdot%
\frac{ s-\lambda _{k}}{s-\lambda _{k-1}}}\cdot
(B_{q_{1},\dots,q_{m}}^{I_{m},s,\lambda _{0},F,n})^{ \frac{\left( \lambda
_{k}-\lambda _{k-1}\right)s }{s-\lambda _{k-1}}}\cdot \Vert T\Vert^{\lambda
_{k-1}\cdot\frac{s-\lambda _{k}}{s-\lambda _{k-1}}}\cdot \Vert T\Vert ^{%
\frac{\left( \lambda _{k}-\lambda _{k-1}\right)s}{ s-\lambda _{k-1}}} \\
& = (B_{q_{1},\dots,q_{m}}^{I_{m},s,\lambda _{0},F,n})^{\lambda_{k}} \Vert
T\Vert ^{\lambda _{k}},
\end{align*}
that is,
\begin{align*}
\left( \sum_{j_{i}=1}^{n}\left( \sum_{\widehat{j_{i}}=1}^{n}\left\vert
T(e_{j_{1}},...,e_{j_{m}})\right\vert ^{s}\right) ^{\frac{1}{s}%
\lambda_{k}}\right) ^{\frac{1}{\lambda_{k}}}& \leq
B_{q_{1},\dots,q_{m}}^{I_{m},s,\lambda _{0},F,n} \Vert T\Vert .
\end{align*}

It remains to consider $k=m$, where $\lambda _{m}=\eta _{1}$. In this case
we have
\begin{align*}
\left( \sum_{j_{i}=1}^{n}\left( \sum_{\widehat{j_{i}}=1}^{n}|T(e_{j_{1}},%
\dots ,e_{j_{m}})|^{s}\right) ^{\frac{1}{s}\eta _{1}}\right) ^{\frac{1}{\eta
_{1}}}& =\left( \sum_{j_{m}=1}^{n}\left( \sum_{\widehat{j_{m}}%
=1}^{n}|T(e_{j_{1}},\dots ,e_{j_{m}})|^{s}\right) ^{\frac{1}{s}\eta
_{1}}\right) ^{\frac{1}{\eta _{1}}} \\
& \leq B_{q_{1},\dots ,q_{m}}^{I_{m},s,\lambda _{0},F,n}\Vert T\Vert ,
\end{align*}%
where the inequality is due to the case $i=k$.

The proof of (b) is similar, except for the last step (case $k=m$), but the
argument is somewhat predictable and we omit the proof.
\end{proof}

\begin{remark}
The case $q_{k}=\infty $ for all $k=1,...,m$ in (a) is known; see, for
instance, \cite{dimant}.
\end{remark}

\section{Main results}

We begin with a technical lemma based on the Contraction Principle (see \cite%
[Theorem 12.2]{Di}). From now on $r_{i}(t)$ are the Rademacher functions.

\begin{lemma}
\label{778}Regardless of the choice of the positive integers $m,N$ and the
scalars $a_{i_{1},\dots ,i_{m}}$, $i_{1},\dots ,i_{m}=1,\dots ,N$,
\begin{equation*}
\max_{\substack{ i_{k}=1,\dots ,N  \\ k=1,\dots ,m}}\left\vert
a_{i_{1},\dots ,i_{m}}\right\vert \leq \int_{\lbrack 0,1]^{m}}\left\vert
\sum_{i_{1},\dots ,i_{m}=1}^{N}r_{i_{1}}(t_{1})\cdots
r_{i_{m}}(t_{m})a_{i_{1},\dots ,i_{m}}\right\vert \,dt_{1}\cdots dt_{m}.
\end{equation*}
\end{lemma}

\begin{proof}
Essentially, one just need to apply the Contraction Principle successively.
We proceed by induction over $m$. The case $m=1$ is precisely the standard
version of Contraction Principle. For all positive integers $i_{1},\dots
,i_{m}$,
\begin{align*}
& \int_{[0,1]^{m}}\left\vert \sum_{i_{1},\dots
,i_{m}=1}^{N}r_{i_{1}}(t_{1})\cdots r_{i_{m}}(t_{m})a_{i_{1},\dots
,i_{m}}\right\vert \,dt_{1}\cdots dt_{m} \\
& =\int_{[0,1]^{m-1}}\left[ \int_{0}^{1}\left\vert
\sum_{i_{1}=1}^{N}r_{i_{1}}(t_{1})\left( \sum_{i_{2},\dots
,i_{m}=1}^{N}r_{i_{2}}(t_{2})\cdots r_{i_{m}}(t_{m})a_{i_{1},\dots
,i_{m}}\right) \right\vert \,dt_{1}\right] \,dt_{2}\cdots dt_{m} \\
& \geq \int_{\lbrack 0,1]^{m-1}}\left\vert \sum_{i_{2},\dots
,i_{m}=1}^{N}r_{i_{2}}(t_{2})\cdots r_{i_{m}}(t_{m})a_{i_{1},\dots
,i_{m}}\right\vert \,dt_{2}\cdots dt_{m} \\
& \geq \left\vert a_{i_{1},\dots ,i_{m}}\right\vert ,
\end{align*}%
where we used the Contraction Principle and the induction hypothesis on the
first and second inequality, respectively. This concludes the proof.
\end{proof}

Now we are able to prove our first main result, providing better constants
for Theorem \ref{i99}:

\begin{theorem}
\label{76543}Let $m\geq 2$ be a positive integer, $p_{j}>1$ for all $j$ and%
\begin{equation*}
\frac{1}{2}\leq \frac{1}{p_{1}}+\cdots +\frac{1}{p_{m}}<1.
\end{equation*}%
Then
\begin{equation*}
\left( \sum_{j_{1},...,j_{m}=1}^{n}\left\vert T\left( e_{j_{1}},\ldots
,e_{j_{m}}\right) \right\vert ^{\frac{1}{1-\left( \frac{1}{p_{1}}+\cdots +%
\frac{1}{p_{m}}\right) }}\right) ^{1-\left( \frac{1}{p_{1}}+\cdots +\frac{1}{%
p_{m}}\right) }\leq 2^{(m-1)\left( 1-\left( \frac{1}{p_{1}}+\cdots +\frac{1}{%
p_{m}}\right) \right) }\left\Vert T\right\Vert
\end{equation*}%
for all $m$-linear forms $T:\ell _{p_{1}}^{n}\times \cdots \times \ell
_{p_{m}}^{n}\rightarrow \mathbb{K}$ and all positive integers $n$. In
particular, if $m<p\leq 2m$ then
\begin{equation*}
\left( \sum_{j_{1},...,j_{m}=1}^{\infty }\left\vert T\left( e_{j_{1}},\ldots
,e_{j_{m}}\right) \right\vert ^{\frac{p}{p-m}}\right) ^{\frac{p-m}{p}}\leq
2^{\frac{\left( m-1\right) \left( p-m\right) }{p}}\left\Vert T\right\Vert
\end{equation*}
for all continuous $m$-linear forms $T:\ell _{p}\times \cdots \times \ell
_{p}\rightarrow \mathbb{K}.$
\end{theorem}

\begin{proof}
Let $S:\ell _{\infty }^{n}\times \cdots \times \ell _{\infty
}^{n}\rightarrow \mathbb{K}$ be an $m$-linear form. Consider $s=\left(
1-\left( \frac{1}{p_{1}}+\cdots +\frac{1}{p_{m}}\right) \right) ^{-1}$.
Since $s\geq 2$, from Lemma \ref{778}, H\"{o}lder's inequality and
Khinchin's inequality for multiple sums (\cite{popa}) we have
\begin{align*}
& \sum_{j_{1}=1}^{n}\left( \sum_{\widehat{j_{1}}=1}^{n}\left\vert S\left(
e_{j_{1}},\ldots ,e_{j_{m}}\right) \right\vert ^{s}\right) ^{\frac{1}{s}} \\
& \leq \sum_{j_{1}=1}^{n}\left( \left( \left( \sum_{\widehat{j_{1}}%
=1}^{n}\left\vert S\left( e_{j_{1}},\ldots ,e_{j_{m}}\right) \right\vert
^{2}\right) ^{\frac{1}{2}}\right) ^{\frac{2}{s}}\left( \max_{\widehat{j_{1}}%
}\left\vert S\left( e_{j_{1}},\ldots ,e_{j_{m}}\right) \right\vert \right)
^{1-\frac{2}{s}}\right) \\
& \leq \sum_{j_{1}=1}^{n}\left( \left( (\sqrt{2})^{m-1}R_{n}\right) ^{\frac{2%
}{s}}R_{n}^{1-\frac{2}{s}}\right) \\
& =2^{(m-1)\left( 1-\left( \frac{1}{p_{1}}+\cdots +\frac{1}{p_{m}}\right)
\right) }\sum_{j_{1}=1}^{n}\int_{[0,1]^{m-1}}\left\vert \sum_{\widehat{j_{1}}%
=1}^{n}r_{j_{2}}(t_{2})\cdots r_{j_{m}}(t_{m})S\left( e_{j_{1}},\ldots
,e_{j_{m}}\right) \right\vert \,dt_{2}\cdots dt_{m} \\
& =2^{(m-1)\left( 1-\left( \frac{1}{p_{1}}+\cdots +\frac{1}{p_{m}}\right)
\right) }\int_{[0,1]^{m-1}}\sum_{j_{1}=1}^{n}\left\vert S\left(
e_{j_{1}},\sum_{j_{2}=1}^{n}r_{j_{2}}(t_{2})e_{j_{2}},\ldots
,\sum_{j_{m}=1}^{n}r_{j_{m}}(t_{m})e_{j_{m}}\right) \right\vert
\,dt_{2}\cdots dt_{m} \\
& \leq 2^{\frac{\left( m-1\right) \left( p-m\right) }{p}}\sup_{t_{2},\ldots
,t_{m}\in \lbrack 0,1]}\sum_{j_{1}=1}^{n}\left\vert S\left(
e_{j_{1}},\sum_{j_{2}=1}^{n}r_{j_{2}}(t_{2})e_{j_{2}},\ldots
,\sum_{j_{m}=1}^{n}r_{j_{m}}(t_{m})e_{j_{m}}\right) \right\vert \\
& \leq 2^{(m-1)\left( 1-\left( \frac{1}{p_{1}}+\cdots +\frac{1}{p_{m}}%
\right) \right) }\sup_{t_{2},\ldots ,t_{m}\in \lbrack 0,1]}\left\Vert
S\left( \ \cdot \ ,\sum_{j_{2}=1}^{n}r_{j_{2}}(t_{2})e_{j_{2}},\ldots
,\sum_{j_{m}=1}^{n}r_{j_{m}}(t_{m})e_{j_{m}}\right) \right\Vert \\
& \leq 2^{(m-1)\left( 1-\left( \frac{1}{p_{1}}+\cdots +\frac{1}{p_{m}}%
\right) \right) }\left\Vert S\right\Vert ,
\end{align*}%
where
\begin{equation*}
R_{n}:=\int_{[0,1]^{m-1}}\left\vert \sum_{\widehat{j_{1}}%
=1}^{n}r_{j_{2}}(t_{2})\cdots r_{j_{m}}(t_{m})S\left( e_{j_{1}},\ldots
,e_{j_{m}}\right) \right\vert \,dt_{2}\cdots dt_{m}.
\end{equation*}%
Repeating the same procedure for other indexes we have%
\begin{equation*}
\sum_{j_{i}=1}^{n}\left( \sum_{\widehat{j_{i}}=1}^{n}\left\vert S\left(
e_{j_{1}},\ldots ,e_{j_{m}}\right) \right\vert ^{s}\right) \leq
2^{(m-1)\left( 1-\left( \frac{1}{p_{1}}+\cdots +\frac{1}{p_{m}}\right)
\right) }\left\Vert S\right\Vert
\end{equation*}%
for all $i=1,...,m.$ Hence, from Lemma \ref{geral}, item (a), we conclude
that
\begin{equation*}
\left( \sum_{j_{1},...,j_{m}=1}^{n}\left\vert T\left( e_{j_{1}},\ldots
,e_{j_{m}}\right) \right\vert ^{s}\right) ^{\frac{1}{s}}\leq 2^{(m-1)\left(
1-\left( \frac{1}{p_{1}}+\cdots +\frac{1}{p_{m}}\right) \right) }\left\Vert
T\right\Vert
\end{equation*}%
for all $m$-linear forms $T:\ell _{p_{1}}^{n}\times \cdots \times \ell
_{p_{m}}^{n}\rightarrow \mathbb{K}$ and all positive integers $n$.
\end{proof}

In particular the above result shows that for $m<p\leq m+c$ for a certain
fixed constant $c$, we have a kind of uniform Hardy--Littlewood inequality,
in the sense that there exists a universal constant, independent of $m$,
satisfying the respective inequalities. For instance, if $c=1$ we have

\begin{equation*}
\left( \sum_{j_{1},...,j_{m}=1}^{\infty }\left\vert T\left( e_{j_{1}},\ldots
,e_{j_{m}}\right) \right\vert ^{\frac{p}{p-m}}\right) ^{\frac{p-m}{p}}\leq
2^{\frac{m-1}{m+1}}\left\Vert T\right\Vert <2\left\Vert T\right\Vert
\end{equation*}%
for all continuous $m$-linear forms $T:\ell _{p}\times \cdots \times \ell
_{p}\rightarrow \mathbb{K}$ with $m<p\leq m+1.$

If $p\leq 2m-2$ we are able to improve exponents and constants:

\begin{theorem}
\label{765432}Let $m\geq 2$ be a positive integer and $m<p\leq 2m-2.$ Then,
for all continuous $m$-linear forms $T:\ell _{p}\times \cdots \times \ell
_{p}\rightarrow \mathbb{K}$, we have
\begin{equation*}
\left( \sum_{j_{i}=1}^{n}\left( \sum_{\widehat{j_{i}}=1}^{n}\left\vert
T\left( e_{j_{1}},\ldots ,e_{j_{m}}\right) \right\vert ^{\frac{p}{p-\left(
m-1\right) }}\right) ^{\frac{p-\left( m-1\right) }{p}\cdot \frac{p}{p-m}%
}\right) ^{\frac{p-m}{p}}\leq 2^{\frac{\left( m-1\right) \left( p-m+1\right)
}{p}}\left\Vert T\right\Vert .
\end{equation*}
\end{theorem}

\begin{proof}
Consider $s=\frac{p}{p-\left( m-1\right) }$. Since $p\leq 2m-2$ we have $%
s\geq 2.$

From Lemma \ref{778}, H\"{o}lder's inequality and Khinchin's inequality for
multiple sums (\cite{popa}) we have, as in the proof of Theorem \ref{76543},
\begin{equation*}
\sum_{j_{i}=1}^{n}\left( \sum_{\widehat{j_{i}}=1}^{n}\left\vert T\left(
e_{j_{1}},\ldots ,e_{j_{m}}\right) \right\vert ^{\frac{p}{p-\left(
m-1\right) }}\right) ^{\frac{p-\left( m-1\right) }{p}}\leq 2^{\frac{\left(
m-1\right) \left( p-m+1\right) }{p}}\left\Vert T\right\Vert
\end{equation*}%
for all $m$-linear forms $T\in \mathcal{L}\left( ^{m}\ell _{\infty }^{n};%
\mathbb{K}\right) $ and all positive integers $n$. Note that
\begin{equation*}
\left[ \frac{1}{\lambda _{0}}-\sum_{j=1}^{m-1}\left( \frac{1}{p_{j}}-\frac{1%
}{q_{j}}\right) \right] ^{-1}=\frac{1}{1-\frac{m-1}{p}}=\frac{p}{p-\left(
m-1\right) }=s.
\end{equation*}

From Lemma \ref{geral}, item (b), we conclude that%
\begin{equation*}
\left( \sum_{j_{i}=1}^{n}\left( \sum_{\widehat{j_{i}}=1}^{n}\left\vert
T\left( e_{j_{1}},\ldots ,e_{j_{m}}\right) \right\vert ^{\frac{p}{p-\left(
m-1\right) }}\right) ^{\frac{p-\left( m-1\right) }{p}\cdot \frac{p}{p-m}%
}\right) ^{\frac{p-m}{p}}\leq 2^{\frac{\left( m-1\right) \left( p-m+1\right)
}{p}}\left\Vert T\right\Vert .
\end{equation*}%
for all continuous $m$-linear forms $T:\ell _{p}\times \cdots \times \ell
_{p}\rightarrow \mathbb{K}$.
\end{proof}

If we have the additional hypothesis that
\begin{equation*}
\frac{1}{2}\leq \frac{1}{p_{1}}+\frac{1}{p_{2}}<1
\end{equation*}%
the optimal constants are always bounded by $2^{1-\left(\frac{1}{p_1}+\frac{1%
}{p_2}\right)}$:

\begin{theorem}
\label{7654321}Let $m\geq 3$ and $p_{1},\ldots ,p_{m}\in (1,\infty ]$ be
such that
\begin{equation*}
\frac{1}{2}\leq \frac{1}{p_{1}}+\frac{1}{p_{2}}<1
\end{equation*}%
and
\begin{equation*}
\frac{1}{p_{1}}+\cdots +\frac{1}{p_{m}}<1.
\end{equation*}%
Then
\begin{equation*}
\left( \sum_{j_{m}=1}^{n}\left( \sum_{\widehat{j_{m}}=1}^{n}|T(e_{j_{1}},%
\ldots ,e_{j_{m}})|^{\frac{1}{1-\left( \frac{1}{p_{1}}+\cdots +\frac{1}{%
p_{m-1}}\right) }}\right) ^{\frac{1-\left( \frac{1}{p_{1}}+\cdots +\frac{1}{%
p_{m-1}}\right) }{1-\left( \frac{1}{p_{1}}+\cdots +\frac{1}{p_{m}}\right) }%
}\right) ^{1-\left( \frac{1}{p_{1}}+\cdots +\frac{1}{p_{m}}\right) }\leq
2^{1-\left( \frac{1}{p_{1}}+\frac{1}{p_{2}}\right) }\Vert T\Vert
\end{equation*}%
for all $m$-linear forms $T:\ell _{p_{1}}^{n}\times \cdots \times \ell
_{p_{m}}^{n}\rightarrow \mathbb{K}$ and all positive integers $n$.
\end{theorem}

\begin{proof}
Since $\frac{1}{p_{1}}+\frac{1}{p_{2}}\geq \frac{1}{2}$ by Theorem \ref%
{76543} we have
\begin{equation*}
\left( \sum_{i,j=1}^{n}|T_{2}(e_{i},e_{j})|^{\frac{1}{1-\left( \frac{1}{p_{1}%
}+\frac{1}{p_{2}}\right) }}\right) ^{1-\left( \frac{1}{p_{1}}+\frac{1}{p_{2}}%
\right) }\leq 2^{1-\left( \frac{1}{p_{1}}+\frac{1}{p_{2}}\right) }\Vert
T\Vert
\end{equation*}%
for all bilinear forms $T_{2}:\ell _{p_{1}}^{n}\times \ell
_{p_{2}}^{n}\rightarrow \mathbb{K}$ and all positive integers $n$. By the
Khinchin inequality we conclude that
\begin{equation*}
\left( \sum\limits_{i,j=1}^{n}\left( \sum\limits_{k=1}^{n}\left\vert
T_{3}\left( e_{i},e_{j},e_{k}\right) \right\vert ^{2}\right) ^{\frac{1}{2}.%
\frac{1}{1-\left( \frac{1}{p_{1}}+\frac{1}{p_{2}}\right) }}\right)
^{1-\left( \frac{1}{p_{1}}+\frac{1}{p_{2}}\right) }\leq 2^{1-\left( \frac{1}{%
p_{1}}+\frac{1}{p_{2}}\right) }\Vert T\Vert
\end{equation*}%
for all $3$-linear forms $T_{3}:\ell _{p_{1}}^{n}\times \ell
_{p_{2}}^{n}\times \ell _{\infty }^{n}\rightarrow \mathbb{K}$ and all
positive integers $n$. In fact, for all positive integers $n$ we have
\begin{align*}
& \left( \sum\limits_{i,j=1}^{n}\left( \sum\limits_{k=1}^{n}\left\vert
T_{3}\left( e_{i},e_{j},e_{k}\right) \right\vert ^{2}\right) ^{\frac{1}{2}.%
\frac{1}{1-\left( \frac{1}{p_{1}}+\frac{1}{p_{2}}\right) }}\right)
^{1-\left( \frac{1}{p_{1}}+\frac{1}{p_{2}}\right) } \\
& \leq A_{\frac{1}{1-\left( \frac{1}{p_{1}}+\frac{1}{p_{2}}\right) }%
}^{-1}\left( \sum\limits_{i,j=1}^{n}\int_{0}^{1}\left\vert
\sum\limits_{k=1}^{n}r_{k}(t)T_{3}\left( e_{i},e_{j},e_{k}\right)
\right\vert ^{\frac{1}{1-\left( \frac{1}{p_{1}}+\frac{1}{p_{2}}\right) }%
}dt\right) ^{1-\left( \frac{1}{p_{1}}+\frac{1}{p_{2}}\right) } \\
& =\left( \int_{0}^{1}\sum\limits_{i,j=1}^{n}\left\vert T_{3}\left(
e_{i},e_{j},\sum\limits_{k=1}^{n}r_{k}(t)e_{k}\right) \right\vert ^{\frac{1}{%
1-\left( \frac{1}{p_{1}}+\frac{1}{p_{2}}\right) }}dt\right) ^{1-\left( \frac{%
1}{p_{1}}+\frac{1}{p_{2}}\right) } \\
& \leq \sup_{t\in \lbrack 0,1]}\left( \sum\limits_{i,j=1}^{n}\left\vert
T_{3}\left( e_{i},e_{j},\sum\limits_{k=1}^{n}r_{k}(t)e_{k}\right)
\right\vert ^{\frac{1}{1-\left( \frac{1}{p_{1}}+\frac{1}{p_{2}}\right) }%
}\right) ^{1-\left( \frac{1}{p_{1}}+\frac{1}{p_{2}}\right) } \\
& \leq 2^{1-\left( \frac{1}{p_{1}}+\frac{1}{p_{2}}\right) }\left\Vert
T_{3}\right\Vert
\end{align*}%
Thus, since $\left( 1-\left( \frac{1}{p_{1}}+\frac{1}{p_{2}}\right) \right)
^{-1}\geq 2$,
\begin{equation*}
\left( \sum\limits_{i,j,k=1}^{n}\left\vert T_{3}\left(
e_{i},e_{j},e_{k}\right) \right\vert ^{\frac{1}{1-\left( \frac{1}{p_{1}}+%
\frac{1}{p_{2}}\right) }}\right) ^{1-\left( \frac{1}{p_{1}}+\frac{1}{p_{2}}%
\right) }\leq 2^{1-\left( \frac{1}{p_{1}}+\frac{1}{p_{2}}\right) }\left\Vert
T_{3}\right\Vert
\end{equation*}%
for all $3$-linear forms $T_{3}:\ell _{p_{1}}^{n}\times \ell
_{p_{2}}^{n}\times \ell _{\infty }^{n}\rightarrow \mathbb{K}$ and all
positive integers $n$. This means that for any Banach spaces $%
E_{1},E_{2},E_{3},$ every continuous $3$-linear form $R:E_{1}\times
E_{2}\times E_{3}\rightarrow \mathbb{K}$ is multiple $\left( \frac{1}{%
1-\left( \frac{1}{p_{1}}+\frac{1}{p_{2}}\right) };p_{1}^{\ast },p_{2}^{\ast
},1\right) $-summing (see \cite{dimant}). By the essence of the inclusion
theorem for multiple summing operators proved in \cite{rp}, since
\begin{equation*}
\frac{1}{1}-\frac{1}{\frac{1}{1-\left( \frac{1}{p_{1}}+\frac{1}{p_{2}}%
\right) }}=\frac{1}{p_{3}^{\ast }}-\frac{1}{\frac{1}{1-\left( \frac{1}{p_{1}}%
+\frac{1}{p_{2}}+\frac{1}{p_{3}}\right) }},
\end{equation*}%
with $E_{1}=\ell _{p_{1}}^{n}$, $E_{2}=\ell _{p_{2}}^{n}$ and $E_{3}=\ell
_{p_{3}}^{n},$ we conclude that
\begin{eqnarray*}
&&\left( \sum\limits_{k=1}^{n}\left( \sum\limits_{i,j=1}^{n}\left\vert
S\left( e_{i},e_{j},e_{k}\right) \right\vert ^{\frac{1}{1-\left( \frac{1}{%
p_{1}}+\frac{1}{p_{2}}\right) }}\right) ^{\left( 1-\left( \frac{1}{p_{1}}+%
\frac{1}{p_{2}}\right) \right) \cdot \frac{1}{1-\left( \frac{1}{p_{1}}+\frac{%
1}{p_{2}}+\frac{1}{p_{3}}\right) }}\right) ^{1-\left( \frac{1}{p_{1}}+\frac{1%
}{p_{2}}+\frac{1}{p_{3}}\right) } \\
&\leq &2^{1-\left( \frac{1}{p_{1}}+\frac{1}{p_{2}}\right) }\left\Vert
S\right\Vert
\end{eqnarray*}%
for all $3$-linear forms $S:\ell _{p_{1}}^{n}\times \ell _{p_{2}}^{n}\times
\ell _{p_{3}}^{n}\rightarrow \mathbb{K}$ and all positive integers $n$. The
proof is completed by a standard induction argument.
\end{proof}


\begin{thebibliography}{99}
\bibitem{jfa-abps} N. Albuquerque, F. Bayart, D. Pellegrino and J.
Seoane--Sep\'{u}lveda, Sharp generalizations of the multilinear
Bohnenblust--Hille inequality, J. Funct. Anal. \textbf{266} (2014),
3726-3740.

\bibitem{isr-abps} N. Albuquerque, F. Bayart, D. Pellegrino and J.
Seoane--Sep\'{u}lveda, Optimal Hardy--Littlewood type inequalities for
polynomials and multilinear operators, Israel Journal of Mathematics \textbf{%
211} (2016), 197-220.

\bibitem{ap} G. Ara\'ujo and D. Pellegrino, On the constants of the
Bohnenblust--Hille and Hardy--Littlewood inequalities, to appear in Bull
Braz. Math. Soc. (2016).

\bibitem{ap2} G. Ara\'ujo, D. Pellegrino and D. D. P. Silva e Silva, On the
upper bounds for the constants of the Hardy--Littlewood inequality. J.
Funct. Anal. \textbf{267}(6) (2014), 1878-1888.

\bibitem{aaps} R. M. Aron, D. N\'{u}\~{n}ez-Alarc\'{o}n, D. Pellegrino and
D. M. Serrano-Rodr\'{\i}guez, Optimal exponents for Hardy--Littlewood
inequalities for $m$-linear operators, arXiv:1602.00178v4.

\bibitem{bayart} F. Bayart, D. Pellegrino and J. Seoane-Sepulveda, The Bohr
radius of the $n$-dimensional polydisk is equivalent to $\sqrt{\frac{\log n}{%
n}}$, Adv. Math. \textbf{264} (2014), 726-746.

\bibitem{bh} H. F. Bohnenblust and E. Hille, On the absolute convergence of
Dirichlet series, Ann. of Math. \textbf{32} (1931), 600-622.

\bibitem{cap} W. Cavalcante and D. N\'{u}\~{n}ez-Alarc\'{o}n, Remarks on an
inequality of Hardy and Littlewood, to appear in Quaest. Math. (2016).

\bibitem{Di} J. Diestel, H. Jarchow and A. Tonge, Absolutely summing
operators, Cambridge University Press, Cambridge, 1995.

\bibitem{dimant} V. Dimant and P. Sevilla-Peris, Summation of coefficients
of polynomials on $\ell _{p}$ spaces, Publ. Mat. \textbf{60} (2016), 289-310.

\bibitem{garling} D. J. H. Garling, Inequalites: A Journey into Linear
Analysis, Cambridge University Press.

\bibitem{hl} G. Hardy and J. E. Littlewood, Bilinear forms bounded in space $%
[p,q]$, Quart. J. Math. \textbf{5} (1934), 241-254.

\bibitem{LLL} J. E. Littlewood, On bounded bilinear forms in an infinite
number of variables, Quart. J. (Oxford Ser.) \textbf{1} (1930), 164-174.

\bibitem{montanaro} A. Montanaro, Some applications of hypercontractive
inequalities in quantum information theory. J. Math. Phys. \textbf{53}(12)
(2012), 122-206.

\bibitem{nnn} D. Pellegrino, The optimal constants of the mixed ($%
\ell_{1},\ell_{2}$)-Littlewood inequality. J. Number Theory \textbf{160}
(2016), 11-18.

\bibitem{pt} D. Pellegrino and E. Teixeira, Towards sharp Bohnenblust--Hille
constants, arXiv:1604.07595v2.

\bibitem{rp} D. Pellegrino, J. Santos, D. M. Serrano-Rodr\'iguez and E.
Teixeira, Regularity principle in sequence spaces and applications,
arXiv:1608.03423.

\bibitem{popa} D. Popa, Multiple Rademacher means and their applications. J.
Math. Anal. Appl. \textbf{386}(2) (2012), 699-708.

\bibitem{pra} T. Praciano-Pereira, On bounded multilinear forms on a class
of $\ell _{p}$ spaces. J. Math. Anal. Appl. \textbf{81}(2) (1981), 561-568.
\end{thebibliography}
\end{document}